\documentclass[11pt]{article}

\usepackage[utf8]{inputenc}
\usepackage{algorithm}
\usepackage{algorithmic}
\usepackage{amsfonts}
\usepackage{amsmath,amssymb}
\usepackage{amsthm} 
\usepackage{booktabs}
\usepackage{color}
\usepackage{enumerate}
\usepackage{graphicx}
\usepackage{latexsym}
\usepackage{longtable}
\usepackage{mathtools}
\usepackage{multirow}
\usepackage{setspace}
\usepackage{subfigure}
\usepackage{url}

\definecolor{myred}{RGB}{255,50,50}         
\definecolor{myblack}{RGB}{0,0,0}           
\newcommand{\red}[1]{\textcolor{myblack}{#1}}

\usepackage[
  backend=biber,
  firstinits=true,
  style=ieee,
  citestyle=numeric-comp,
  sorting=nyt]{biblatex}

\newtheorem{theorem}{Theorem}[section]
\newtheorem{lemma}[theorem]{Lemma}
\newtheorem{definition}[theorem]{Definition}
\newtheorem{corollary}[theorem]{Corollary}
\newtheorem{proposition}[theorem]{Proposition}

\numberwithin{equation}{section}

\def\usebfsetcapital{\def\setcapital##1{\mathbf{##1}}}

\usebfsetcapital

\def\setR{\mathbb{R}}

\newcommand\condition[1]{\quad \text{#1}}
\newcommand\forallcondition[1]{\condition{for all~$#1$}}

\DeclareMathOperator*{\argmax}{argmax}
\DeclareMathOperator*{\argmin}{argmin}

\DeclareMathOperator{\dom}{dom}

\DeclarePairedDelimiter{\norm}{\lVert}{\rVert}

\DeclarePairedDelimiter{\set}{\lbrace}{\rbrace}
\DeclarePairedDelimiterX{\Set}[2]{\lbrace}{\rbrace}{#1\mathrel{}\delimsize\vert\mathrel{}#2}
\DeclarePairedDelimiterX{\innerp}[2]{\langle}{\rangle}{#1, #2}
\newcommand{\T}{\top\hspace{-1pt}}

{}

\newcommand{\inner}[2]{\langle#1,#2\rangle} 

\begin{document}


\title{Monotonicity for Multiobjective Accelerated\\ Proximal Gradient Methods%
  \thanks{This work was supported by the Grant-in-Aid for Scientific
    Research (C) (19K11840 and 21K11769) from Japan Society for the
    Promotion of Science.}
}
\author{
  Yuki Nishimura$^\dag$
  \and 
  Ellen H. Fukuda$^\dag$
  \and
  Nobuo Yamashita%
  \thanks{Department of Applied Mathematics and Physics, Graduate
    School of Informatics, Kyoto University, Kyoto \mbox{606--8501}, Japan
    (\texttt{yuki.nishimura@amp.i.kyoto-u.ac.jp, ellen@i.kyoto-u.ac.jp,
      nobuo@i.kyoto-u.ac.jp}).}
}

\maketitle


\begin{abstract}
  \noindent  Accelerated proximal gradient methods, which are also called fast
  iterative shrinkage-thres\-holding algorithms (FISTA) are known to
  be efficient for many applications. Recently, Tanabe et al. proposed
  an extension of FISTA for multiobjective optimization
  problems. However, similarly to the single-objective minimization
  case, the objective functions values may increase in some
  iterations, and inexact computations of subproblems can also lead to
  divergence. Motivated by this, here we propose a variant of the
  FISTA for multiobjective optimization, that imposes some
  monotonicity of the objective functions values. In the
  single-objective case, we retrieve the so-called MFISTA, proposed by
  Beck and Teboulle. We also prove that our method has global
  convergence with rate $O(1/k^2)$, \red{where $k$ is the number of
    iterations,} and show some numerical advantages in requiring
  monotonicity.\\

\end{abstract}


In this paper, we consider the following multiobjective optimization
problem:
\begin{equation}
  \label{eq: MOP}
  \tag{MOP}
  \begin{aligned}
    \min & \quad F(x) \\
    \mbox{s.t.} & \quad x \in \setR^n,
  \end{aligned}
\end{equation}
with $F \coloneqq (F_1, \ldots, F_m)^\T$ being a vector-valued
function defined by
\[
F_i(x) \coloneqq f_i(x) + g_i(x), \quad i = 1,\ldots,m,
\]
where $\T$ denotes transpose, $f_i \colon \setR^n \to \setR$ is convex
and continuously differentiable, and $g_i \colon \setR^n \to
(-\infty,\infty]$ is closed, proper and convex. Some applications for
  \eqref{eq: MOP} include problems in image processing area, robust
  optimization and machine learning \red{(see some examples
  in~\cite{BMTB20,FW14,GZY16,Sta88})}.  In particular, if each $g_i$ is
  the indicator function of a same set, then the problem \eqref{eq:
    MOP} is equivalent to a constrained optimization
  problem~\cite{TFY19}.

As it is well known, the simultaneous minimization of multiple
objectives is done by using the concept of Pareto optimality. In this
work, we are particularly interested in solving~\eqref{eq: MOP} in the
Pareto sense, and by using a multiobjective descent
method~\cite{FGD14}. The research associated to multiobjective descent
methods is currently increasing, as alternatives to
metaheuristics~\cite{JMT02}, where no theoretical guarantee of
convergence exists, and also scalarization
approaches~\cite{Jah84,Luc87}, where unknown parameters are
required. Many of these methods have been proposed in the literature,
including for instance, the steepest descent~\cite{FS00}, the
Newton~\cite{FGDS09}, the projected gradient~\cite{FGD11,GDI04}, the
proximal point~\cite{BIS05}, the subgradient~\cite{CNSFL}, and the
conjugate gradient~\cite{LPP18} methods.

Recently, Tanabe et al. proposed the so-called multiobjective
\emph{proximal gradient method} (PGM), that is specific for composite
problems like~\eqref{eq: MOP} \cite{TFY19}. Afterwards, an accelerated
version of it was also proposed~\cite{TFY22b}. Such methods are
basically extensions of the ones proposed for single-valued
optimization, given in~\cite{FM81,BT09a}. In particular, the
accelerated version of PGM, which is often called \emph{fast iterative
shrinkage-thresholding algorithm} (FISTA), is used in many
applications of signal processing area. \red{While many enhancements
on the scalar-valued FISTA had been proposed, most of their
multiobjective counterparts had not; currently, only the original
FISTA had been extended to the vector-valued
context}~\cite{TFY22b}. Here, we fill this gap, by proposing a variant
of the multiobjective~FISTA.

More specifically, the multiobjective FISTA does not guarantee
monotonically decrease of the objective functions
values~\cite{TFY22b}. \red{This property is inherited from the
  single-objective FISTA.} However, as it was demonstrated
in~\cite{BT09b}, for some problems, inexact computations of
subproblems may lead to a high nonmonotonicity, and even
divergence. To overcome \red{the above} drawback, a monotone version
of FISTA, called MFISTA, was proposed for scalar-valued
problems~\cite{BT09b}. Based on \red{that} work, here we propose a
multiobjective version of MFISTA. \red{To this end}, we first clarify
the notion of monotonicity in the multiobjective case. More precisely,
in each iteration, we consider either the case that all objectives
decrease, or when at least one objective function decreases. For the
latter case, we show that the method converges globally with rate
$O(1/k^2)$, \red{where $k$ is the number of iterations}, which is the
same rate of the multiobjective FISTA.

The paper is organized as follows. In Section~\ref{sec:preliminaries},
we give some notations, and remark the multiobjective proximal
gradient methods, with and without acceleration. We propose our
monotone version of FISTA in Section~\ref{sec:monotone_fista}, and its
convergence analysis is given in Section~\ref{sec:convergence}. Moreover,
in Section~\ref{sec:experiments} we show some numerical
experiments. We conclude the paper in Section~\ref{sec:conclusions} with
some final remarks.


\section{Preliminaries}
\label{sec:preliminaries}

Let us first present some basic notations that will be used throughout
the paper. For given vectors $x, y \in \setR^s$, we write $x \le y$
(resp. $x < y$) if $x_i \le y_i$ (resp. $x_i < y_i$) for all
$i=1,\dots,s$. The Euclidean norm and inner product are denoted by
$\norm{\,\cdot\,}$, and $\inner{\cdot}{\cdot}$, respectively. For a
given matrix $A \in \setR^{r \times s}$, its transpose is given by
$A^\top \in \setR^{s \times r}$. The gradient and the subdifferential
of a function $\phi\colon\setR^s \to (-\infty, \infty]$ at $x \in
  \setR^s$ are written as $\nabla \phi(x)$, and $\partial \phi(x)$,
  respectively.

Let us now return to problem~\eqref{eq: MOP}. In the whole paper, we
assume that each $f_i$ has a Lipschitz continuous gradient with
constant $L_i > 0$. Since we deal with various objective functions, we
also define
\begin{equation}
  \label{eq: lipchitz}
  L \coloneqq \max_{i \in \{1,\ldots, m\}} L_i.
\end{equation}
We also denote the domain of the objective function~$F$ as follows:
\begin{equation}
  \text{dom} F \coloneqq \{x \in \setR^n \mid F (x) < \infty\}.
\end{equation}

As it is well known, \red{in general, we may not have a} point that
minimizes all objectives at once. Thus, \eqref{eq: MOP} is solved in
the Pareto sense. Recall that if there is no $\red{x} \in \setR^n$
such that $F(x) \leq F(x^*)$ and $F(\red{x}) \neq F(x^*)$, then $x^*$
is a (strongly) Pareto optimal point. Moreover, if there is no
$\red{x} \in \setR^n$ such that $F(x) < F(x^*)$, then $x^*$ is called
weakly Pareto optimal point. In this work, we are particularly
interested in the this latter optimality concept.

Let us now review the multiobjective FISTA, proposed in~\cite{TFY22b},
which is the base of our work. As \red{do most descent methods, this
  one updates a solution in an iterative manner}, and at each
iteration the following subproblem is solved, with given $x,y \in
\setR^n$:
\begin{equation}
  \label{prob: placc}
  \min_{z \in \setR^n} \max_{i \in \{1,\ldots,m\}}
  \Big( \innerp*{\nabla f_i(y)}{z-y} + g_i(z) + f_i(y) - F_i(x) \Big)
  + \frac{\ell}{2} \norm*{z-y}^2,
\end{equation}
where $\ell \ge L$. Besides the quadratic regularization term, for
each $i$, it basically considers the first-order approximation of the
objective, and additional \red{terms that are} important for the
acceleration process. In particular, when $m=1$, this is equivalent to
the subproblem used in the single-objective FISTA~\cite{BT09a}.
Moreover, the subproblem~\eqref{prob: placc} has a unique optimal
solution, because its objective function is strongly convex. Thus,
denote such solution as
\begin{equation}
  \label{eq: subp_sol}
  d_\ell(x,y) := \argmin_{z \in \setR^n} \max_{i \in \{1,\ldots,m\}}
  \Big( \innerp*{\nabla f_i(y)}{z-y} + g_i(z) + f_i(y) - F_i(x) \Big)
  + \frac{\ell}{2} \norm*{z-y}^2.
\end{equation}
  
Furthermore, consider the merit function below, proposed
in~\cite{TFY20}:
\begin{align}
  u_0(x) &\coloneqq \sup_{z\in \setR^n} \min_{i \in \{1,\ldots,m\}}
  \Big( F_i(x)-F_i(z) \Big).
  \label{def: merit}
\end{align}
Note that when $m=1$, $u_0(x)$ just measures the distance between
$F(x)$ and the \red{optimal objective} value, which means that it is
the simplest merit function for single-objective problems. The
following result shows that the weak Pareto optimality can be
characterized in terms of the solution of the subproblem and the above
merit function.

\begin{proposition}
  \label{prop: merit}
  Let $d_\ell$ and $u_0$ be defined by~\eqref{eq: subp_sol}
  and~\eqref{def: merit}, respectively. Then, the following conditions
  are equivalent:
  \begin{itemize}
  \item[(a)] $y \in \setR^n$ is weakly Pareto optimal for~\eqref{eq: MOP},
  \item[(b)] $d_\ell(x,y) = y$ for some $x \in \setR^n$,
  \item[(c)] $u_0(y) = 0$.
  \end{itemize}
  Moreover, $u_0(x) \ge 0$ for all $x \in \setR^n$ and the function
  $d_\ell$ is continuous.
\end{proposition}

\begin{proof}
  See~\cite[Theorem~3.1]{TFY20} and~\cite[Proposition~4.1]{TFY22b}.
\end{proof}

We end this section by stating below the algorithm proposed
in~\cite{TFY22b}. As we can see in~\cite[Theorem~5.2]{TFY22b}, this
method converges globally with rate $O(1/k^2)$, retrieving the rate of
the single-objective FISTA~\cite[Theorem~4.4]{BT09a}.

\begin{algorithm}[hbtp]
  \caption{Multiobjective FISTA}
  \label{alg: acc-pgm}
  \begin{algorithmic}[1]
    \REQUIRE $y^1 = x^0 \in \text{dom} F,\ t_1 = 1,\ \ell \ge L,\ \varepsilon > 0,\ k := 1$.
    \ENSURE $x^\ast$: weakly Pareto optimal solution
    \WHILE{$\norm*{d_\ell(x^{k - 1}, y^k) - y^k} \ge \varepsilon$} \vspace{3pt}
    \STATE $x^k := d_\ell(x^{k - 1}, y^k)$ \vspace{3pt}
    \STATE $t_{k+1} := \Big( 1+\sqrt{1 + 4t_k^2} \, \Big) / 2$ \vspace{3pt}
    \STATE $y^{k + 1} := x^k + \displaystyle{\left(\frac{t_k-1}{t_{k+1}}\right) (x^k - x^{k - 1})}$ 
    \STATE $k := k + 1$
    \ENDWHILE
  \end{algorithmic}
\end{algorithm}

Moreover, differently from the non-accelerated proximal gradient
method \cite[Sections~3.1 and 3.2]{TFY19}, we do not necessarily have
$F_i(x^k) \le F_i(x^{k-1})$ for each $i=1,\dots,m$ in (multiobjective)
FISTA. For some instances of problems, and depending on the precision
for the computation of the subproblems, the lack of monotonicity may
lead to divergence. \red{To resolve this issue}, in the next section,
we propose multiobjective FISTA algorithms with some monotonicity
properties, that can be seem as extensions of the single-objective
MFISTA~\cite{BT09b}.


\section{The proposed method}
\label{sec:monotone_fista}

Here, we establish the concept \red{of} monotonicity in multiobjective
optimization, and then propose two multiobjective FISTA algorithms
with monotonicity.

\begin{definition}
  \label{def: monotone}
  Let $\set{x^k}$ be a sequence generated by an algorithm applied for
  problem~\eqref{eq: MOP}. For a given iterate~$k$,
  \begin{itemize}
  \item[(a)] If $\displaystyle{\min_{i \in \{1,\ldots,m\}}
    \Big(F_i(x^{k-1})-F_i(x^k) \Big)\geq 0}$,
    then we say that $F$ is strongly decreasing.
  \item[(b)] If $\displaystyle{\max_{i \in \{1,\ldots,m\}}
    \Big(F_i(x^{k-1})-F_i(x^k) \Big)\geq 0}$,
    then we say that $F$ is weakly decreasing.
  \end{itemize}
\end{definition}

From the definition (a), strongly decreasing $F$ at iteration $k$
means basically that $F_i(x^{k-1}) \ge F_i(x^k)$ for all $i$, i.e.,
all objective functions values are not increasing in that
iteration. If such condition holds in all iterations, we can say that
the method generates a sequence with monotonically decreasing
functional values. \red{We can also impose a stronger condition by
using strict inequality in the definition (a) (which can be implied
by the existence of the so-called admissible curve~\cite{Sma73}).}
However, one can think that \red{(a) and its strict version are too
  strong conditions}. For this reason, we also consider the
monotonicity defined in~(b). In such case, $F$ is weakly decreasing if
there exists some~$j$ such that $F_j(x^{k-1}) \ge F_j(x^k)$, or in
other words, at least one objective function does not increase in that
iteration. However, since such~$j$ may not be the same for all
iterations, it is not necessarily true that some objective function
decreases monotonically. In the subsequent analysis, we will see that
the condition (b) allows global convergence of the method, and for
such reason we state below the algorithm using only such condition.

\begin{algorithm}[hbtp]
  \caption{Multiobjective MFISTA}
  \label{alg: mfista}
  \begin{algorithmic}[1]
    \REQUIRE Set $y^1 = x^0 \in \dom F,\ t_1 = 1,\ \ell \ge L,\ \varepsilon > 0,\ k := 1$.
    \ENSURE $x^\ast$: weakly Pareto optimal solution
    \WHILE{$\norm*{d_\ell(x^{k - 1}, y^k) - y^k} \ge \varepsilon$} \vspace{3pt}
    \STATE $z^k := d_\ell(x^{k - 1}, y^k)$ \vspace{3pt}
    \IF{$\displaystyle \max_{i \in \{1,\ldots,m\}}\Big(F_i(x^{k-1})-F_i(z^k) \Big)\geq 0$,} 
    \STATE $x^k := z^k$
    \ELSE
    \STATE $x^k := x^{k-1}$
    \ENDIF
    \STATE $t_{k+1} := \Big( 1+\sqrt{1 + 4t_k^2} \, \Big) / 2$\vspace{3pt}
    \STATE $y^{k+1} := \displaystyle{x^k + \left(\frac{t_k}{t_{k+1}}\right) (z^k - x^k)
      +  \left(\frac{t_k-1}{t_{k+1}}\right) (x^k - x^{k-1})}$ 
    \STATE $k := k + 1$
    \ENDWHILE
  \end{algorithmic}
\end{algorithm}

Comparing to Algorithm~\ref{alg: acc-pgm}, we observe that a new
variable $z^k$ is used here. From Step~3, if the monotonicity
condition (in this case, from Definition~\ref{def: monotone}(b))
holds, then we proceed as in Algorithm~\ref{alg: acc-pgm}. On the
other hand, if all objective functions values increase, then $y^k$ is
updated as a convex combination with the previous point $x^{k-1} =
x^k$ and the one computed with the subproblem, i.e., $z^k =
d_\ell(x^{k - 1}, y^k)$. As it can be seen in~\cite{BT09b}, when
$m=1$, Algorithm~\ref{alg: mfista} retrieves the single-valued MFISTA.
We end this section with a simple inequality that holds in all
iterations, and that will be used in the convergence analysis.

\begin{lemma}
  \label{lem: F}
  Let $\{x^k\}$ be generated by Algorithm~\ref{alg: mfista}. Then, for all~$k \ge 0$,
  the following inequality holds:
  \begin{equation*}
    \min_{i \in \{1,\ldots,m\}} \Big( F_i(z^k) - F_i(x^k)\Big) \ge 0.
  \end{equation*}
\end{lemma}

\begin{proof}
  In each iteration of the algorithm, either $x^k = z^k$ (Step~4) or
  $x^k = x^{k-1}$ (Step~6) hold.  If the former is satisfied, then
  trivially, $F_i(z^k) - F_i(x^k) = 0$ for all $i=1,\ldots,m$, and so
  the inequality holds. On the other hand, if the latter holds,
  from Step~3 we have
  \[
  \max_{i \in \{1,\ldots,m\}}\Big(F_i(x^{k-1})-F_i(z^k) \Big) \ngeq 0,
  \]
  which means $F_i(z^k) > F_i(x^{k-1})$ for all $i = 1,\ldots,m$.
  Since $x^k = x^{k-1}$, we also have $F_i(x^k) =
  F_i(x^{k-1})$. Therefore, the claim is also true in this case.
\end{proof}


\section{Convergence analysis}
\label{sec:convergence}

In this section, we will prove that the proposed multiobjective MFISTA
(Algorithm~\ref{alg: mfista}) converges globally with rate
$O(1/k^2)$. Before this, let us recall that we assume Lipschitz
continuity of the gradients $\nabla f_i$. Then, from~\eqref{eq:
  lipchitz}, for all $p, q \in \setR^n$ and~$i = 1, \ldots, m$, we
have
\[
f_i(p) - f_i(q) \le \innerp*{\nabla f_i(q)}{p - q} + \frac{L}{2} \norm*{p - q}^2.
\]
This inequality is often called descent lemma~\cite[Proposition
  A.24]{Ber99}. This, together with the definition of $F_i$\red{,} gives
\begin{equation}
  \label{eq: descent mid}
  \begin{split}
    F_i(p) - F_i(r) &= f_i(p) - f_i(q) + g_i(p) + f_i(q) - F_i(r) \\
    &\le \innerp*{\nabla f_i(q)}{p - q} + g_i(p) + f_i(q) - F_i(r)
    + \frac{L}{2} \norm*{p - q}^2
  \end{split}
\end{equation}
for all~$p,\ q,\ r \in \setR^n$ and $i = 1, \ldots, m$. 

Returning to the algorithm itself, from the optimality conditions of
the subproblem~\eqref{prob: placc} (more specifically, that obtains
$d_\ell(x^k,y^{k+1})$), there exist $\mu(x^k,y^{k+1}) \in \partial
g(d_\ell(x^k,y^{k+1}))$ and $\lambda(x^k,y^{k+1}) \in \setR^m$, with
$\lambda_i(x^k,y^{k+1}) \geq 0$ for all~$i=1,\dots,m$ such that
\begin{subequations}
  \label{eq: kkt}
  \begin{gather} 
    \sum_{i = 1}^m \lambda_i(x^k, y^{k+1}) \big( \nabla f_i(y^{k+1}) + \mu_i(x^k, y^{k+1}) \big)
    = - \ell \left( d_\ell(x^k, y^{k+1}) - y^{k+1} \right), \label{eq: optimal} \\
    \sum_{i = 1}^m \lambda_i(x^k, y^{k+1}) = 1, \quad \lambda_j(x^k, y^{k+1})
    = 0 \forallcondition{j \notin \mathcal{A}(x^k, y^{k+1})}, \label{eq: lambda}
    \end{gather}
\end{subequations}
where $\mathcal{A}(x^k, y^{k+1})$ is the following set of active indices:
\begin{align*}
  \mathcal{A}(x^k, y^{k+1}) \coloneqq & \argmax_{i \in \{1, \ldots, m\}}
  \Big( \innerp*{\nabla f_i(y^{k+1})}{d_\ell(x^k, y^{k+1}) - y^{k+1}} \\
  & \hspace{50pt} + g_i(d_\ell(x^k, y^{k+1})) + f_i(y^{k+1}) - F_i(x^k) \Big).
\end{align*}
Furthermore, for all $k \ge 0$, we define $\sigma_k \colon \setR^n \to
[-\infty, \infty)$ as follows:
\begin{equation}
  \label{eq: sigma}
  \sigma_k(z) \coloneqq \min_{i \in \{ 1, \ldots, m \}}\Big( F_i(x^k) - F_i(z) \Big).
\end{equation}
Observe that this function is actually the objective function of the
problem that defines~$u_0(x^k)$, given in~\eqref{def: merit}.
In the subsequent analysis, we will give an upper bound
for~$\sigma_k$, but first we establish the following technical lemma.

\begin{lemma}
  \label{lem: sigma}
  Let $\{x^k\}$, $\{y^k\}$ and $\{z^k\}$ be sequences generated by
  Algorithm~\ref{alg: mfista}. Then, for all $z \in \setR^n$ and~$k
  \ge 0$ we have:
  \begin{itemize}
  \item[(a)] $\displaystyle{-\sigma_{k + 1}(z) \ge \frac{\ell}{2} \left(
    \norm*{z^{k + 1}}^2 - \norm*{y^{k + 1}}^2 -2 \innerp*{z^{k+1} - y^{k + 1}}{z} \right)
    + \frac{\ell - L}{2} \norm*{z^{k + 1} - y^{k + 1}}^2}$,
  \item[(b)] $\sigma_k(z) - \sigma_{k + 1}(z)$\\[3pt]
    $\: \displaystyle{\ge \frac{\ell}{2} \left(
    2 \innerp*{z^{k+1} - y^{k + 1}}{y^{k + 1} - x^k} + \norm*{z^{k + 1} - y^{k + 1}}^2 \right) 
    + \frac{\ell - L}{2} \norm*{z^{k + 1} - y^{k + 1}}^2}$.
  \end{itemize}
\end{lemma}

\begin{proof}
  (a) Let $z \in \setR^n$ and $k \ge 0$. From the definition of $\sigma_{k+1}$ in~\eqref{eq: sigma} and
  using Lemma~\ref{lem: F}, we have
  \[
  \begin{split}
    \sigma_{k + 1}(z) & = \min_{i \in \{1, \ldots, m\}} \left( F_i(x^{k + 1}) - F_i(z) \right) \\
    &\le \min_{i \in \{ 1, \ldots, m\}} \left( F_i(z^{k + 1}) - F_i(x^{k+1}) \right)
    + \min_{i \in \{ 1, \ldots, m\}} \left( F_i(x^{k + 1}) - F_i(z) \right) \\
    &\le \min_{i \in \{ 1, \ldots, m\}} \left( F_i(z^{k + 1}) - F_i(z) \right),
  \end{split}
  \]
  where the second inequality holds because for all $p, q \in  \setR^m$,
  \begin{equation}
    \label{eq: min_relations}
    \displaystyle\min_{i \in \{ 1, \ldots, m\}} p_i + \min_{i \in \{ 1, \ldots, m\}} q_i
    \le \min_{i \in \{ 1, \ldots, m\}}(p_i + q_i).
  \end{equation}
  From~\eqref{eq: lambda}, since $\lambda_i(x^k, y^{k+1}) \ge 0$ and
  $\sum_{i=1}^m \lambda_i(x^k, y^{k+1}) = 1$, we obtain
  \[  
  \sigma_{k + 1}(z) \le \sum_{i = 1}^m \lambda_i(x^k, y^{k + 1})
  \left( F_i(z^{k + 1}) - F_i(z) \right).
  \]
  Now, using~\eqref{eq: descent mid} with $p = z^{k + 1}$, $q = y^{k + 1}$
  and $r = z$, we get
  \[
  \begin{split}
    \sigma_{k + 1}(z)
    & \le \sum_{i = 1}^m \lambda_i(x^k, y^{k + 1}) \left(
    \innerp*{\nabla f_i(y^{k + 1})}{z^{k + 1} - y^{k + 1}}
    + g_i(z^{k + 1}) + f_i(y^{k + 1}) - F_i(z) \right) \\
    & \:\quad + \frac{L}{2} \norm*{z^{k + 1} - y^{k + 1}}^2.
  \end{split}
  \]
  Also, the convexity of $f_i$ and $g_i$ yield
  \[
  \begin{split} 
    & \sigma_{k + 1}(z)\\
    & \begin{multlined}
        \le \sum_{i = 1}^m \lambda_i(x^k, y^{k + 1}) \left (
        \innerp*{\nabla f_i(y^{k + 1})}{z^{k + 1} - y^{k + 1}}
        + \innerp*{\nabla f_i(y^{k + 1})}{y^{k + 1} - z} \right. \\
        + \left. \innerp*{\mu_i(x^k, y^{k + 1})}{z^{k + 1} - z} \right)
        + \frac{L}{2} \norm*{z^{k + 1} - y^{k + 1}}^2
     \end{multlined} \\
    & = \sum_{i = 1}^m \lambda_i(x^k, y^{k + 1})
    \innerp*{\nabla f_i(y^{k + 1}) + \mu_i(x^k, y^{k + 1})}{z^{k + 1} - z}
    + \frac{L}{2} \norm*{z^{k + 1} - y^{k + 1}}^2.
  \end{split}
  \]
  From Step 2 of Algorithm~\ref{alg: mfista}, we know that
  $d_{\ell}(x^k,y^{k+1}) = z^{k+1}$, and using~\eqref{eq: optimal} we obtain
  \[
  \sigma_{k + 1}(z) \le - \ell \innerp*{z^{k+1} - y^{k + 1}}{z^{k + 1} - z}
  + \frac{L}{2} \norm*{z^{k + 1} - y^{k + 1}}^2.
  \]
  Therefore, we get
  \[
  \begin{split}
    & \sigma_{k + 1}(z) \\
    & \le - \frac{\ell}{2} \left( 2 \innerp*{z^{k+1} - y^{k + 1}}{z^{k + 1} - z}
    - \norm*{z^{k + 1} - y^{k + 1}}^2 \right) - \frac{\ell - L}{2} \norm*{z^{k + 1} - y^{k + 1}}^2 \\
    & = - \frac{\ell}{2} \left( \norm*{z^{k + 1}}^2 - \norm*{y^{k + 1}}^2 -2
    \innerp*{z^{k+1} - y^{k + 1}}{z} \right) - \frac{\ell - L}{2} \norm*{z^{k + 1} - y^{k + 1}}^2,
  \end{split}
  \]
  which shows that (a) holds.\\
  
  \noindent (b) Once again, from the definition of $\sigma_k$
  in~\eqref{eq: sigma} and Lemma~\ref{lem: F}, we have
  \[
  \begin{split}
    & \sigma_k(z) - \sigma_{k + 1}(z) \\
    & \ge \min_{i \in \{ 1, \ldots, m\}}\left( F_i(x^k) - F_i(z) \right)
    - \min_{i \in \{ 1, \ldots, m\}}\left( F_i(x^{k + 1}) - F_i(z) \right)\\
    & \:\quad - \min_{i \in \{ 1, \ldots, m\}}\left( F_i(z^{k+1}) - F_i(x^{k+1}) \right)\\
    & \ge \min_{i \in \{ 1, \ldots, m\}}\left( F_i(x^k) - F_i(z) \right)
    - \min_{i \in \{ 1, \ldots, m\}}\left( F_i(z^{k + 1}) - F_i(z) \right) \\
    &\ge - \max_{i \in \{1, \ldots, m\}}\left( F_i(z^{k + 1}) - F_i(x^k) \right),
  \end{split}
  \]
  where the second and the third inequalities hold using~\eqref{eq: min_relations}.
  From~\eqref{eq: descent mid} with $p = z^{k + 1}$, $q = y^{k+1}$ and $r = x^{k}$,
  we get 
  \[
  \begin{split}
    \MoveEqLeft
    \sigma_k(z) - \sigma_{k + 1}(z) \\
    &\begin{multlined}
       \ge - \max_{i = 1, \ldots, m} \left(
       \innerp*{\nabla f_i(y^{k + 1})}{z^{k + 1} - y^{k + 1}} + g_i(z^{k + 1}) + f_i(y^{k + 1})  \right. \\
       - \left. F_i(x^k) \right) - \frac{L}{2} \norm*{z^{k + 1} - y^{k + 1}}^2
     \end{multlined} \\
    &\begin{multlined}
       = - \sum_{i = 1}^m \lambda_i(x^k, y^{k + 1}) \left(
       \innerp*{\nabla f_i(y^{k + 1})}{z^{k + 1} - y^{k + 1}} + g_i(z^{k + 1}) \right. \\
       + \left. f_i(y^{k + 1}) - F_i(x^k) \right) - \frac{L}{2} \norm*{z^{k + 1} - y^{k + 1}}^2 
     \end{multlined} \\
    &\begin{multlined}
       = - \sum_{i = 1}^m \lambda_i(x^k, y^{k + 1}) \left(
       \innerp*{\nabla f_i(y^{k + 1})}{x^k - y^{k + 1}} + f_i(y^{k + 1}) - f_i(x^k) \right) \\
       - \sum_{i = 1}^m \lambda_i(x^k, y^{k + 1}) \left(
       \innerp*{\nabla f_i(y^{k + 1})}{z^{k + 1} - x^k} + g_i(z^{k + 1}) - g_i(x^k) \right) \\
       - \frac{L}{2} \norm*{z^{k + 1} - y^{k + 1}}^2,
     \end{multlined}
  \end{split}
  \]
  where the first equality holds from~\eqref{eq: lambda} and the second one is true
  by setting $z^{k + 1} - y^{k + 1} = (x^k - y^{k + 1}) + (z^{k + 1} - x^k)$. Since $f_i$ is convex,
  the first term of the above inequality is nonnegative. Moreover, from the convexity
  of $g_i$, we have:
  \begin{align*}
    & \sigma_k(z) - \sigma_{k + 1}(z) \\
    & \ge - \sum_{i = 1}^m {\lambda_i(x^k, y^{k + 1})
      \innerp*{\nabla f_i(y^{k + 1}) + \mu_i(x^k, y^{k + 1})}{z^{k + 1} - x^k}}
    - \frac{L}{2} \norm*{z^{k + 1} - y^{k + 1}}^2,
  \end{align*}
  where $\mu_i(x^k, y^{k + 1}) \in \partial g(d_\ell(x^k,y^{k+1}))$.
  From~\eqref{eq: optimal} the inequality below holds: 
  \[
  \begin{split}
    & \sigma_k(z) - \sigma_{k + 1}(z) \\
    & \ge \ell \innerp*{z^{k+1} - y^{k + 1}}{z^{k + 1} - x^k}
    - \frac{L}{2} \norm*{z^{k + 1} - y^{k + 1}}^2 \\
    & = \frac{\ell}{2} \left( 2 \innerp*{z^{k+1} - y^{k + 1}}{z^{k + 1} - x^k}
    - \norm*{z^{k + 1} - y^{k + 1}}^2 \right) + \frac{\ell - L}{2} \norm*{z^{k + 1} - y^{k + 1}}^2 \\
    & = \frac{\ell}{2} \left( 2 \innerp*{z^{k+1} - y^{k + 1}}{y^{k + 1} - x^k}
    + \norm*{z^{k + 1} - y^{k + 1}}^2 \right) + \frac{\ell - L}{2} \norm*{z^{k + 1} - y^{k + 1}}^2,
  \end{split}
  \]
  which completes the proof. 
\end{proof}

The following result, related to stepsizes $t_k$, is also necessary in
our analysis.
\begin{lemma}
  \label{lem: stepsizes}
  Let $\{t_k\}$ be defined by Algorithm~\ref{alg: mfista}.  Then, for
  all $k \ge 1$, the following assertions hold:
  \[
  (a) \: t_k \ge \frac{k + 1}{2}, \quad
  (b) \: t_k^2 - t_{k+1}^2 + t_{k+1} = 0 \quad \mbox{and} \quad
  (c) \: 1 - \left(\frac{t_{\red{k}} - 1}{t_{k+1}}\right)^2 \ge \frac{1}{t_k} > 0.
  \]
\end{lemma}

\begin{proof}
  See \cite[Lemma 4.2]{TFY22b}.
\end{proof}

Now, we prove that for each $i=1,\dots,m$, the functional values of
any sequence generated by the method do not exceed the value at the
initial point. Note that this result is not trivial, because as
(multiobjective) FISTA, Algorithm~\ref{alg: mfista} also do not
\red{necessarily} guarantee monotonicity for all the objective
functions in all iterations.

\begin{theorem}
  \label{thm: leq ini}
  Let $\set*{x^k}$ be a sequence generated by Algorithm~\ref{alg: mfista}. Then,
  for all $i = 1, \ldots, m$ and $k \ge 0$, we have
  \[
  F_i(x^k) \le F_i(x^0).
  \]
\end{theorem}

\begin{proof} 
  Let $i = 1, \ldots, m$ and $k \ge 0$. Note that
  \[
  \begin{split}
    F_i(x^k) - F_i(x^{k+1})
    & \ge - \max_{i \in \{1, \ldots,m\}} \Big(F_i(x^{k+1}) - F_i(x^k)\Big) \\
    & \ge  - \max_{i \in \{1, \ldots,m\}} \Big(F_i(x^{k+1}) - F_i(x^k)\Big)
    + \max_{i \in \{1, \ldots,m\}} \Big(F_i(x^{k+1}) - F_i(z^{k+1})\Big) \\
    & \ge  - \max_{i \in \{1, \ldots,m\}} \Big(F_i(z^{k+1}) - F_i(x^k)\Big),
  \end{split}
  \]
  where the second inequality holds from Lemma~\ref{lem: F}. Also,
  using the same arguments from Lemma~\ref{lem: sigma}'s proof, we get
  \[
  \begin{split}
    F_i(x^k) - F_i(x^{k+1})\ge \frac{\ell}{2}
    \left( 2\innerp*{z^{k+1}-y^{k+1}}{y^{k+1} - x^k} + \norm*{z^{k+1} - y^{k+1}}^2\right) \\
    + \frac{\ell - L}{2} \norm*{z^{k+1} - y^{k+1}}^2.
  \end{split}
  \]
  Since $\ell \ge L$, we can write
  \[
  \begin{split}
    F_i(x^k) - F_i(x^{k+1})
    & \ge \frac{\ell}{2} \left( 2\innerp*{z^{k+1}-y^{k+1}}{y^{k+1} - x^k}
    + \norm*{z^{k+1} - y^{k+1}}^2\right) \\
    & \ge \frac{\ell}{2} \left( \norm*{z^{k+1} - x^k}^2 - \norm*{y^{k+1} - x^k}^2\right).
  \end{split}
  \]
  From the definition of $y^{k+1}$ in Algorithm~\ref{alg: mfista},
  \begin{equation}
    \label{eq: sub_F}
    F_i(x^k) - F_i(x^{k+1}) \ge \frac{\ell}{2} \left( \norm*{z^{k+1} - x^k}^2
    - \norm*{\frac{t_k}{t_{k+1}}(z^k-x^k) + \frac{t_k-1}{t_{k+1}}(x^k - x^{k-1})}^2\right).
  \end{equation}
  The second term of the right-hand side of the above inequality can be written as
  \begin{align*}
    & - \norm*{\frac{t_k}{t_{k+1}}(z^k-x^k) + \frac{t_k-1}{t_{k+1}}(x^k - x^{k-1})}^2 \\
    = & -\frac{1}{t_{k+1}^2} \norm*{(t_k-1)(z^k-x^{k-1}) + z^k - x^k}^2 \\
    = & - \left[ \left(\frac{t_k-1}{t_{k+1}}\right)^2\norm*{z^k-x^{k-1}}^2
      + \frac{1}{t_{k+1}^2}\norm*{z^k - x^k}^2
      + 2 \left(\frac{t_k-1}{t_{k+1}^2} \right) \innerp*{z^k - x^{k-1}}{z^k-x^k}\right].
  \end{align*}
  Now, let $K \coloneqq \{k \colon x^k = x^{k-1}\}$. Summing up $k$ from
  $k=1$ to $k=\hat{k}$ in~\eqref{eq: sub_F}, we obtain
  \begin{equation*} 
    \begin{split} 
      F_i(x^{\hat{k}+1})
      \le & \: F_i(x^1) - \frac{\ell}{2} \norm*{z^{\hat{k}+1} - x^{\hat{k}}}^2
      - \frac{\ell}{2} \sum_{k=2}^{\hat{k}} \left[ \left(1
      - \left(\frac{t_k-1}{t_{k+1}}\right)^2\right) \norm*{z^k - x^{k-1}}^2\right] \\
      & {} +\frac{\ell}{2} \left(\frac{t_1-1}{t_2}\right) \norm*{z^1 - x^0}^2 \\
      & {} + \frac{\ell}{2} \sum_{k=1, k \in K}^{\hat{k}} \left[ \frac{1}{t_{k+1}^2}
      \norm*{z^k-x^k}^2 + 2 \left( \frac{t_k-1}{t_{k+1}^2} \right)
      \innerp*{z^k - x^{k-1}}{z^k-x^k}\right],
    \end{split}
  \end{equation*}
  where the last summation follows because $k \notin K$ means $x^k =
  z^k$ (see Steps 4 and 6 of Algorithm~\ref{alg: mfista}). Moreover,
  the definition of $K$ and the fact that $t_1 = 1$ give
  \begin{equation*} 
    \begin{split}       
      F_i(x^{\hat{k}+1}) \le & \: F_i(x^1) - \frac{\ell}{2}
      \norm*{z^{\hat{k}+1} - x^{\hat{k}}}^2
      - \frac{\ell}{2} \sum_{k=2}^{\hat{k}} \left[ \left(1 -
      \left(\frac{t_k-1}{t_{k+1}}\right)^2\right) \norm*{z^k - x^{k-1}}^2\right] \\
      & {} + \frac{\ell}{2} \sum_{k=1, k \in K}^{\hat{k}}
      \left[\left(\frac{2t_k-1}{t_{k+1}^2}\right) \norm*{z^k-x^k}^2\right] \\
      \le & \: F_i(x^1) - \frac{\ell}{2} \sum_{k=2, k \in K}^{\hat{k}} \left[ \left(1 -
      \left(\frac{t_k-1}{t_{k+1}}\right)^2\right) \norm*{z^k - x^k}^2\right] \\
      & {} + \frac{\ell}{2} \sum_{k=1, k \in K}^{\hat{k}}
      \left[\left(\frac{2t_k-1}{t_{k+1}^2}\right) \norm*{z^k-x^k}^2\right], \\
    \end{split}
  \end{equation*}
  where the second inequality holds from the non-negativity of the
  norm, the fact that $\ell > 0$, and using Lemma~\ref{lem:
    stepsizes}(c) (only for the case $k \notin K$). Simple
  calculations with Lemma~\ref{lem: stepsizes}(b) shows that
  \begin{equation}
    \label{eq: final_iter}
    \begin{split} 
      F_i(x^{\hat{k}+1}) \le & \: F_i(x^1) - \frac{\ell}{2}
      \sum_{k=2, k \in K}^{\hat{k}} \frac{1}{t_{k+1}} \norm*{z^k - x^k}^2
      + \frac{\ell}{2t_2^2} \norm*{z^1-x^1}^2 \\ 
      \le & \: F_i(x^1) + \frac{\ell}{2t_2^2} \norm*{z^1-x^1}^2.
    \end{split}
  \end{equation}
  Furthermore, note that $z^1$ is the solution of the following subproblem: 
  \begin{equation*}
    \min_{z \in \setR^n} \max_{i \in \{1,\ldots,m\}} \{\innerp*{\nabla f_i(x^0)}{z-x^0}
    + g_i(z) - g_i(x^0)\} + \frac{\ell}{2} \norm*{z-x^0}^2.
  \end{equation*}
  This problem is equivalent to the \red{first iteration's subproblem
    of the} multiobjective proximal gradient
  method. From~\cite{TFY19}, the objective functions values decrease
  monotonically in this case. \red{In particular,} for all $i = 1,\ldots,m$,
  we have
  \begin{equation}
    \label{eq: monotone_pgm}
    F_i(z^1) \le F_i(x^0).
  \end{equation}
  Thus, $x^1 = z^1$ and
  \begin{equation} \label{eq: Fxkhat}
    F_i(x^{\hat{k}+1}) \le F_i(x^1) \quad \mbox{for all } i = 1,\ldots,m
  \end{equation}
  \red{holds from~\eqref{eq: final_iter}. Finally, from \eqref{eq:
    monotone_pgm} and the fact that $x^1=z^1$ we have}
  \begin{equation*}
    F_i(x^k) \le F_i(x^0) \quad \mbox{for all } i = 1,\ldots,m, 
  \end{equation*}
  and the proof is complete.
\end{proof}

We finally give an upper bound on $\sigma_k$ in terms of $t_k$ and the
initial point $x^0$. Then, this result will be used in the main
theorem, related to the complexity of Algorithm~\ref{alg: mfista}.

\begin{lemma}
  \label{lem: key relation}
  Let $\set*{x^k}$ be a sequence generated by Algorithm~\ref{alg: mfista}. Then,
  for all $i = 1, \ldots, m$ and $\ k \ge 0$, we have
  \begin{equation}
    t_k^2\sigma_k(z) \le \frac{\ell}{2} \norm*{x^0 - z}^2.
  \end{equation}
\end{lemma}

\begin{proof}
  Using Lemma~\ref{lem: sigma}, we obtain the following inequalities: 
  \begin{gather}
    \begin{multlined}
      -\sigma_{k + 1}(z) \ge \frac{\ell}{2} \left( \norm*{z^{k + 1}}^2
      - \norm*{y^{k + 1}}^2 -2 \innerp*{z^{k+1} - y^{k + 1}}{z} \right)
      + \frac{\ell - L}{2} \norm*{z^{k + 1} - y^{k + 1}}^2,
    \end{multlined}  \\
    \begin{multlined}
      \sigma_k(z) - \sigma_{k + 1}(z) \ge \frac{\ell}{2} \left(
      2 \innerp*{z^{k+1} - y^{k + 1}}{y^{k + 1} - x^k}
      + \norm*{z^{k + 1} - y^{k + 1}}^2 \right) \\
      + \frac{\ell - L}{2} \norm*{z^{k + 1} - y^{k + 1}}^2.
    \end{multlined}
  \end{gather}
  From Lemma~\ref{lem: stepsizes}(a), $t_k \ge 1$ for all $k$. We can then
  multiply the second inequality above by $(t_{k+1}-1)$, and add it
  to the first inequality. Thus, we get
  \[
  \begin{split}
    & (t_{k+1}-1)\sigma_k(z) - t_{k+1}\sigma_{k+1}(z) \\
    & \ge \frac{\ell}{2} \left( t_{k+1}\norm*{z^{k + 1} - y^{k + 1}}^2
    + 2 \innerp*{z^{k+1} - y^{k + 1}}{t_{k+1}y^{k + 1} - (t_{k+1}-1)x^k - z}\right) \\
    & \:\quad + \frac{\ell-L}{2} \left( t_{k+1} \norm*{z^{k + 1} - y^{k + 1}}^2 \right).
  \end{split}
  \]
  Multiplying the above inequality by $t_{k+1}$, we further obtain
  \[
  \begin{split}
    & t_k^2\sigma_k(z) - t_{k+1}^2\sigma_{k+1}(z) \\
    & \ge \frac{\ell}{2} \left( \norm*{t_{k+1}(z^{k + 1} - y^{k + 1})}^2
    + 2t_{k+1} \innerp*{z^{k+1} - y^{k + 1}}{t_{k+1}y^{k + 1} - (t_{k+1}-1)x^k - z} \right) \\
    &  \:\quad + \frac{\ell-L}{2} t_{k+1}^2 \norm*{z^{k + 1} - y^{k + 1}}^2,
    \end{split}
  \]
  where the first inequality holds from Lemma~\ref{lem: stepsizes}(b). The
  Pythagoras relation
  \[
  \norm*{b-a}^2 + 2 \innerp*{b-a}{a-c} = \norm*{b-c}^2 - \norm*{a-c}^2,
  \]
  with $a \coloneqq t_{k+1}y_{k+1}$, $b \coloneqq t_{k+1}z^{k+1}$ and
  $c \coloneqq (t_{k+1} - 1)x^k + z$ also shows that
  \[
    \begin{split}
      & t_k^2\sigma_k(z) - t_{k+1}^2\sigma_{k+1}(z)\\
      & \ge \frac{\ell}{2} \left( \norm*{t_{k+1}z^{k+1} - (t_{k+1} - 1)x^k - z}^2
      - \norm*{t_{k+1}y^{k + 1} - (t_{k+1} - 1)x^k - z}^2\right) \\
      &  \:\quad + \frac{\ell-L}{2} t_{k+1}^2 \norm*{z^{k + 1} - y^{k + 1}}^2 \\
      & \ge \frac{\ell}{2} \left( \norm*{t_{k+1}z^{k+1} - (t_{k+1} - 1)x^k - z}^2
      - \norm*{t_{k+1}y^{k + 1} - (t_{k+1} - 1)x^k - z}^2 \right).
    \end{split}
  \]
  Now, define $\rho_k \colon \setR^n \to \setR$ as 
  \begin{equation}
    \label{eq: rho}
    \rho_k(z) \coloneqq \norm*{t_{k + 1} z^{k+1} - (t_{k + 1} - 1) x^k - z}^2.
  \end{equation}
  From the definition of $y^{k+1}$ given in Step~9 of Algorithm~\ref{alg: mfista},
  we have
  \begin{equation*}
    t_k^2\sigma_k(z) - t_{k+1}^2\sigma_{k+1}(z) \ge \frac{\ell}{2}
    \left\{ \rho_k(z) - \rho_{k-1}(z)\right\}.
  \end{equation*}
  Let $\hat{k} \ge 1$. Summing the above inequality from $k=1$ to $k = \hat{k}$, we get
  \begin{equation*}
    t_1^2\sigma_1(z) - t_{\hat{k}+1}^2\sigma_{\hat{k}+1}(z)
    \ge \frac{\ell}{2} \left\{ \rho_{\hat{k}}(z) - \rho_0(z)\right\}.
  \end{equation*}
  Furthermore, since $\rho_{\hat{k}}(z) \ge 0$, we obtain
  \begin{equation*}
    t_{\hat{k}+1}^2\sigma_{\hat{k}+1}(z) \le \frac{\ell}{2} \rho_0(z) + t_1^2\sigma_1(z).
  \end{equation*}
  From Lemma~\ref{lem: sigma} and the fact that $t_1=1$, we can write
  \[
  \begin{split}
    t_{\hat{k}+1}^2\sigma_{\hat{k}+1}(z)
    & \le \frac{\ell}{2} \norm*{z^1 - z}^2 + \sigma_1(z)  \\
    & \le \frac{\ell}{2} \norm*{z^1 - z}^2 - \frac{\ell}{2} \left( 2\innerp*{z^1-y^1}{z^1-z}
    - \norm*{z^1-y^1}^2\right) - \frac{\ell-L}{2} \norm*{z^1-y^1}^2.
  \end{split}
  \]
  Because $y^1 = x^0$, we have
  \[
  \begin{split}
    t_{\hat{k}+1}^2\sigma_{\hat{k}+1}(z)
    & \le \frac{\ell}{2} \norm*{z^1 - z}^2 - \frac{\ell}{2} \left( 2\innerp*{z^1-x^0}{z^1-z}
    - \norm*{z^1-x^0}^2\right) - \frac{\ell-L}{2} \norm*{z^1-x^0}^2 \\
    & \le \frac{\ell}{2} \norm*{z^1 - z}^2 - \frac{\ell}{2} \left( 2\innerp*{z^1-x^0}{z^1-z}
    - \norm*{z^1-x^0}^2\right) \\
    & = \frac{\ell}{2} \norm*{x^0 - z}^2,
  \end{split}
  \]
  which shows that $t_{k}^2\sigma_k(z) \le \ell/2 \norm*{x^0 -
    z}^2$, as it is claimed.
\end{proof}

\begin{theorem}
  \label{thm: mfista conv rate}
  Let $X^*$ be the set of weakly Pareto optimal points associated to
  problem~\eqref{eq: MOP}, and assume that it is nonempty. For all $x
  \in \Omega_F(F(x^0)) \coloneqq \{x \in \setR^n \mid F(x) \le
  F(x^0)\}$, assume that there exists $x^* \in X^*$ such that $F(x^*)
  \le F(x)$, and define
  \[
  \begin{split}
    R &\coloneqq \sup_{F^\ast \in F(X^\ast \cap \Omega_F(F(x^0)))}
    \min_{x \in F^{-1}(\set*{F^\ast})} \norm*{x-x^0}^2 < \infty.
  \end{split}
  \]
  Let $\set*{x^k}$ be a sequence generated by Algorithm~\ref{alg: mfista}.
  Then, for all $k \ge 0$, we have
  \[
  u_0(x^k) \le \frac{2\ell R}{(k+1)^2}.
  \] 
\end{theorem}

\begin{proof}
  Let $k \ge 0$ and take $z \in \setR^n$ arbitrarily. From the
  Lemma~\ref{lem: stepsizes}(a), we know that $t_k \ge
  (k+1)/2$. Therefore, from Lemma~\ref{lem: key relation}, we have
  \[
  \begin{split}
    \sigma_k(z)
    = \min_{i = 1, \ldots, m}\left( F_i(x^k) - F_i(z) \right)
    \le \frac{\ell}{2t_k^2} \norm*{x^0 - z}^2  
    \le \frac{2\ell\norm*{x^0-z}^2}{(k+1)^2},
  \end{split}
  \]
  and thus,
  \[
  \sup_{F^\ast \in F(X^\ast \cap \Omega_F(F(x^0)))} \min_{z \in F^{-1}(\set*{F^\ast})}
  \left\{F_i(x^k) - F_i(z)\right\} \le \frac{2\ell R}{(k+1)^2}.
  \]
  Using Theorem~\ref{thm: leq ini} and with similar arguments used in
  the proof of~\cite[Theorem~5.2]{TFY22a}, we obtain
  \[
  u_0(x^k) \le \frac{2\ell R}{(k + 1)^2},
  \]
  and the proof is complete.
\end{proof}

Recalling Proposition~\ref{prop: merit}, the above result states that
the Algorithm~\ref{alg: mfista} converges globally (in the weakly
Pareto sense) with convergence rate $O(1/k^2)$. Here, we use the same
scalar~$R$ used in the convergence analysis of multiobjective
FISTA~\cite[Assumption~3.1]{TFY22b}, which was in turn used in the
analysis of the multiobjective proximal gradient
method~\cite[Assumption~5.1]{TFY22a}. As stated in
\cite[Remark~5.3]{TFY22a}, the assumption holds trivially when $m=1$
and at least one optimal point exists, or when the level set
$\Omega_F(F(x^0))$ is bounded. 

\begin{corollary}
  Suppose that the same assumptions of Theorem~\ref{thm: mfista conv rate}
  hold, and let $\set*{x^k}$ be a sequence generated by
  Algorithm~\ref{alg: mfista}. Then, every accumulation point of $\set*{x^k}$
  is weakly Pareto optimal for~\eqref{eq: MOP}.
\end{corollary}

\begin{proof}
  It is clear from Theorem~\ref{thm: mfista conv rate},
  Proposition~\ref{prop: merit} and the lower-semicontinuity of the
  objectives.
\end{proof}

\noindent For completeness, we state above the immediate global
convergence result. Naturally, if each objective function is also
strictly convex, we also end up in (strongly) Pareto optimal
points~\cite[Theorem~3.1]{FGDS09}.


\section{Numerical experiments}
\label{sec:experiments}

In this section, we present some simple numerical experiments to
validate the described results. We modified the code used
in~\cite{TFY22b}\footnote{The source code is available in
  \url{https://github.com/zalgo3/zfista}}, implemented the proposed
method in Python~3.7.4 and ran all the experiments on a 1.1~GHz Intel
Core i5 machine with 4 cores and 8GB of memory.  Besides the results
of Algorithm~\ref{alg: mfista}, which we simply state as Weak-MFISTA
here, we also show the results of the multiobjective proximal gradient
method (PGM)~\cite{TFY19}, multiobjective FISTA~\cite{TFY22b}, and the
version of MFISTA replacing the weakly decrasing condition
(Definition~\ref{def: monotone}(b)) with the strongly decreasing one
(Definition~\ref{def: monotone}(a)), which we call Strong-MFISTA.  We
consider the following test problems~\cite{FGDS09}, with and without
objective~$g$.\\

\noindent \textbf{Problem 1.} $m=3$, and
\begin{align*}
  & f_1(x) = \frac{1}{n^2} \sum_{i = 1}^n i (x_i - i)^4,
  \quad f_2(x) = \exp \left( \sum_{i = 1}^n \frac{x_i}{n} \right) + \|x\|_2^2, \\
  & f_3(x) = \frac{1}{n(n + 1)} \sum_{i = 1}^n i (n - i + 1) \exp (- x_i), \\
  & g_1(x) = g_2(x) =  g_3(x) = 0. 
\end{align*}

\noindent \textbf{Problem 2.} $m=3$, and
\begin{align*}
  & f_1(x), f_2(x), \mbox{and } f_3(x) \mbox{ defined as in Problem 1}, \\
  & g_1(x) = g_2(x) =  g_3(x) = \chi_{\setR_+^n}(x),
\end{align*}
where $\chi_{\setR_+^n}$ denotes the indicator function of the set
$\setR_+^n$.\\

For \red{the above problems}, we run all the algorithms $100$ times, with
different initial points, which are taken in the intervals $[-2,2]^n$
and $[0,2]^n$ for Problem~1 and Problem~2, respectively. The stopping
parameter $\varepsilon$ is set as $10^{-5}$, and the dimension $n$ is
equal to $10$. Also, the subproblems are converted to their dual and
solved with a trust-region interior point method using Scipy
library. The final solutions are shown in Figure~\ref{fig: answer}. In
each problem, we obtain similar \red{sets of weakly Pareto optimal
points} for all versions of FISTA methods.

\begin{figure}[tbp]
  \begin{tabular}{@{\hspace{-3pt}}c} 
    \begin{minipage}[b]{0.5\linewidth}
      \centering
      \includegraphics[keepaspectratio, scale=0.38]{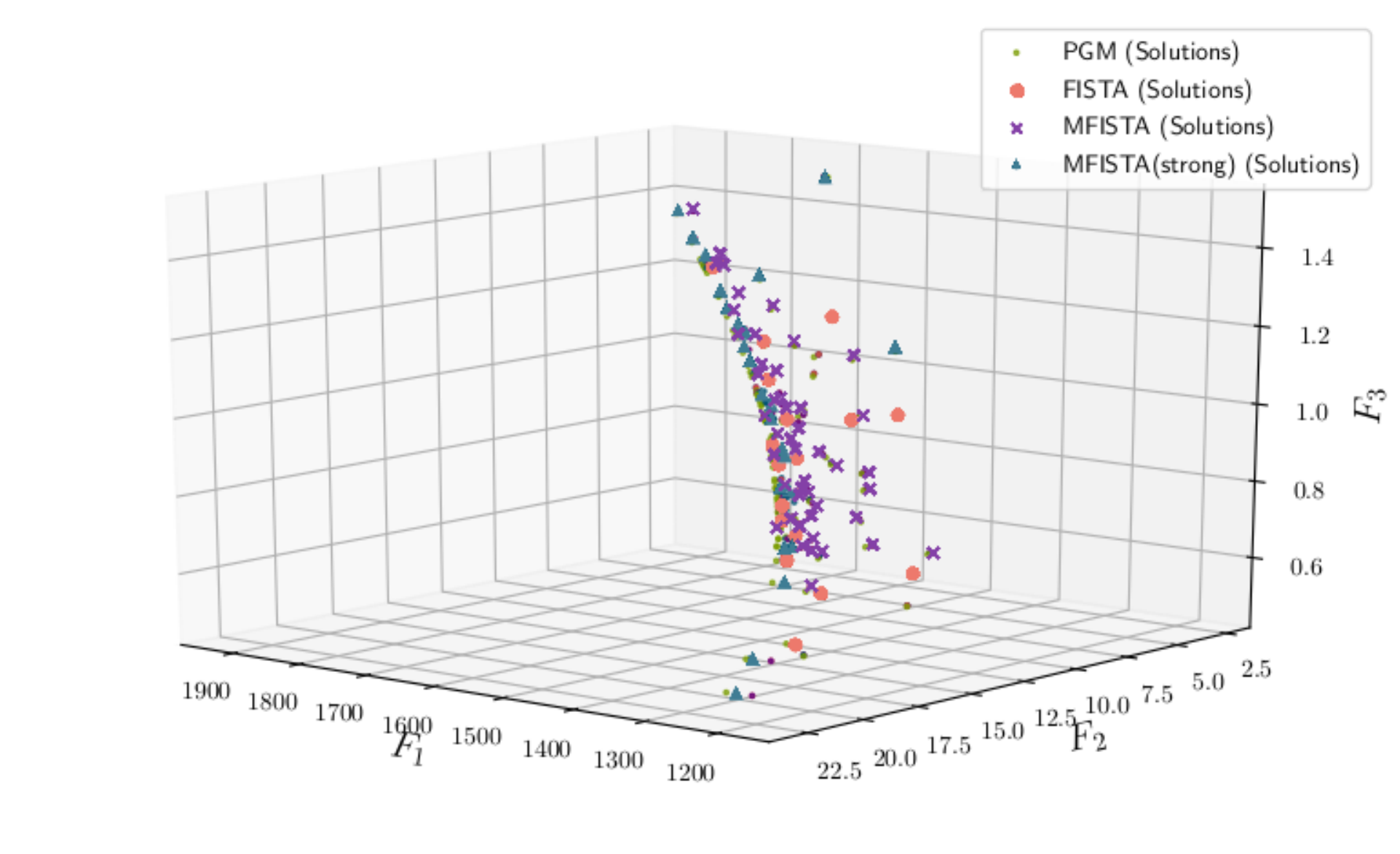}
    \end{minipage}
    \begin{minipage}[b]{0.5\linewidth}
      \centering
      \includegraphics[keepaspectratio, scale=0.38]{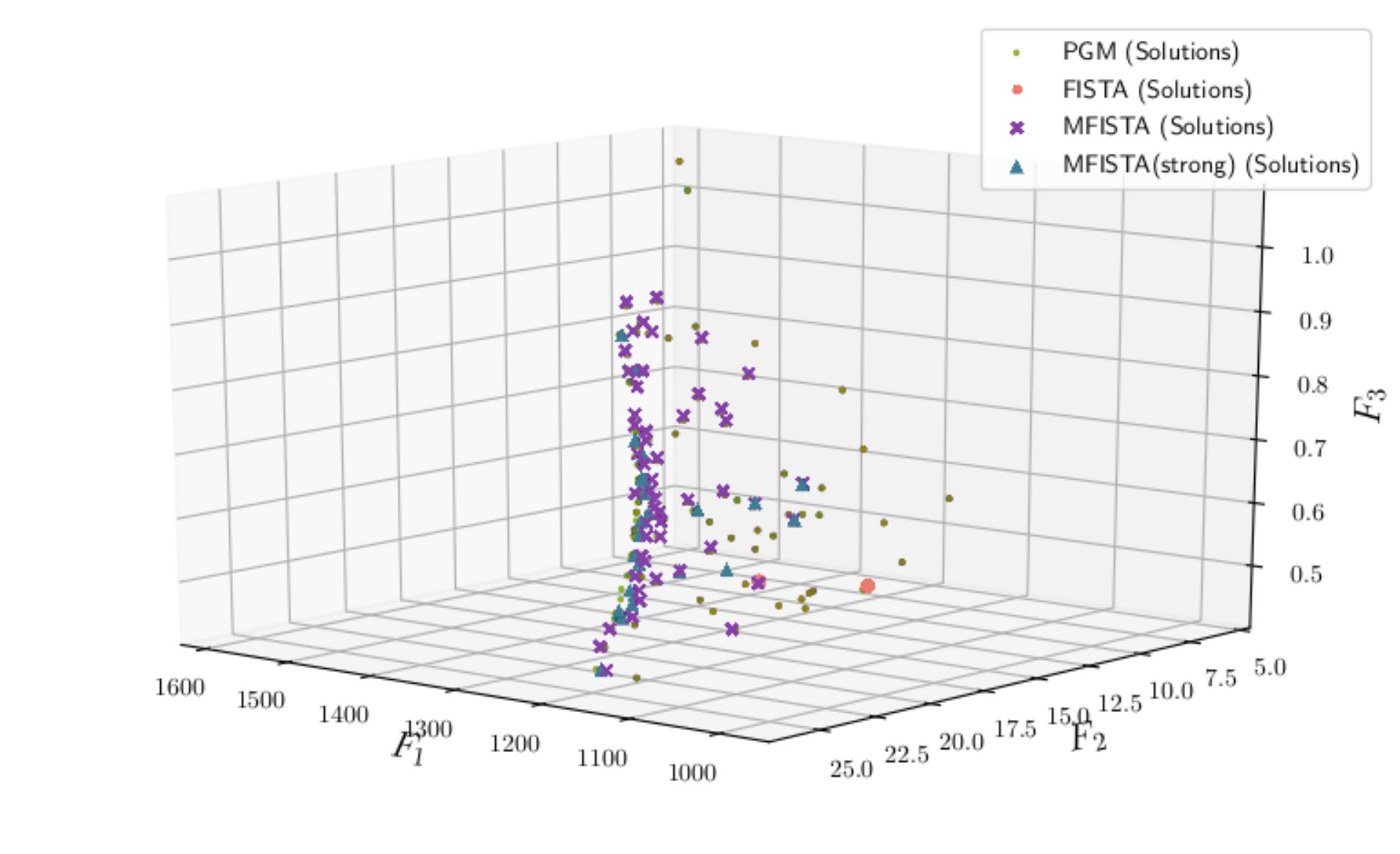}
    \end{minipage}
  \end{tabular}
  \caption{Pareto solutions obtained for Problem 1 (left) and Problem 2 (right)}
  \label{fig: answer}
\end{figure}

Moreover, in Figures~\ref{fig: FGDS09} and~\ref{fig: FGDS09_cons}, for
a given initial point, we plot $|F_i(x^k) - F_i(x^*)|$, where $x^*$ is
the final (possible solution) point, at each iteration~$k$. The
logarithm is taken just for better visibility. As it is possible to
see, for Strong-MFISTA, the objective functions are nonincreasing for
all iterations. On the other hand, for Problem~1 and differently from
FISTA, the Weak-MFISTA shows increase for $F_2$ and $F_3$, but
decreases for at least the objective function~$F_1$. A similar
situation happens for Problem~2, but the difference between FISTA and
Weak-MFISTA is small in this case.

\begin{figure}[htbp]
  \begin{center}
    \begin{tabular}{@{\hspace{-3pt}}c} 
      \begin{minipage}[b]{0.33\linewidth}
        \centering
        \includegraphics[keepaspectratio, width=50mm]{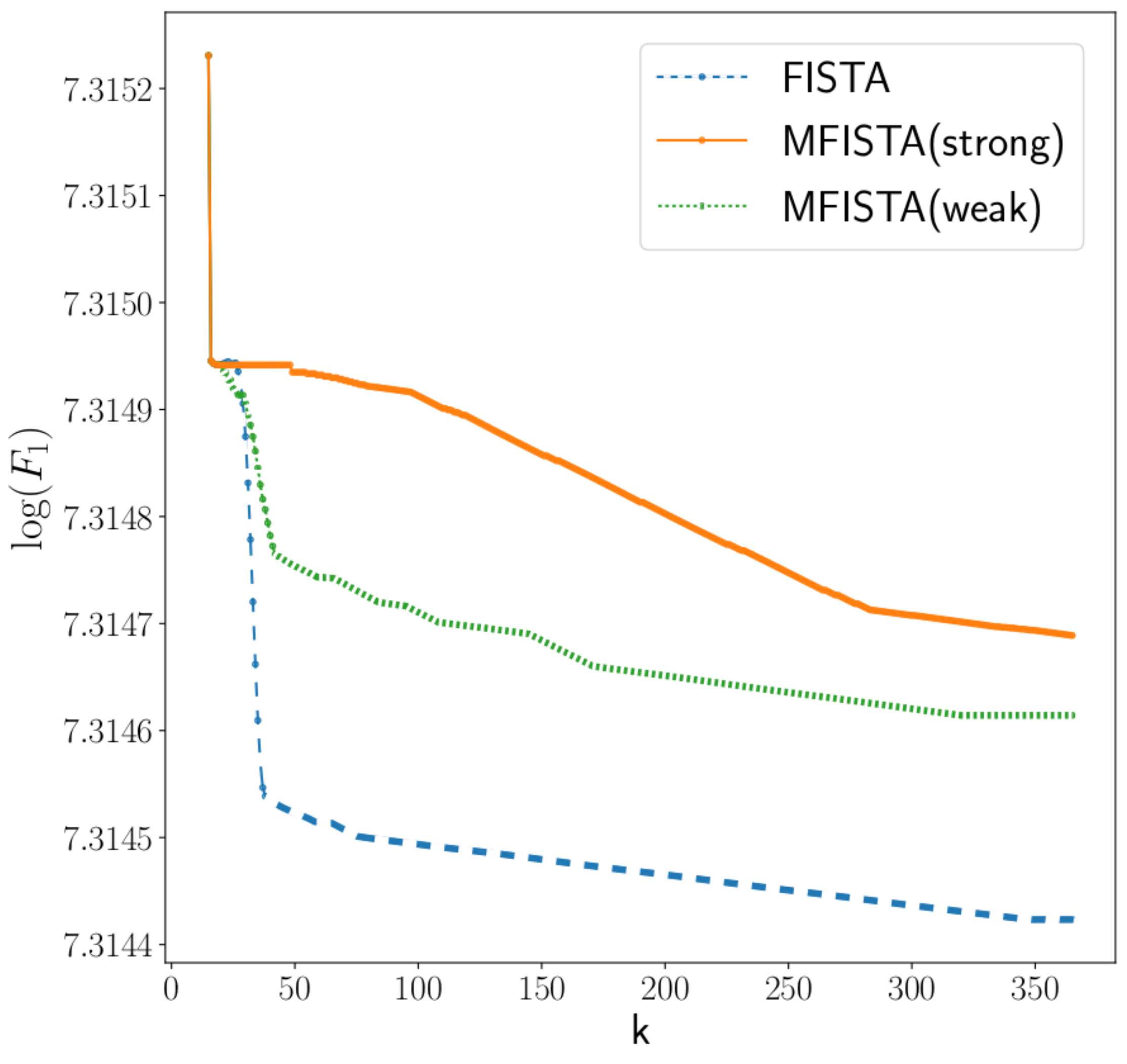}
        $F_1$
      \end{minipage}
      \begin{minipage}[b]{0.33\linewidth}
        \centering
        \includegraphics[keepaspectratio, width=50mm]{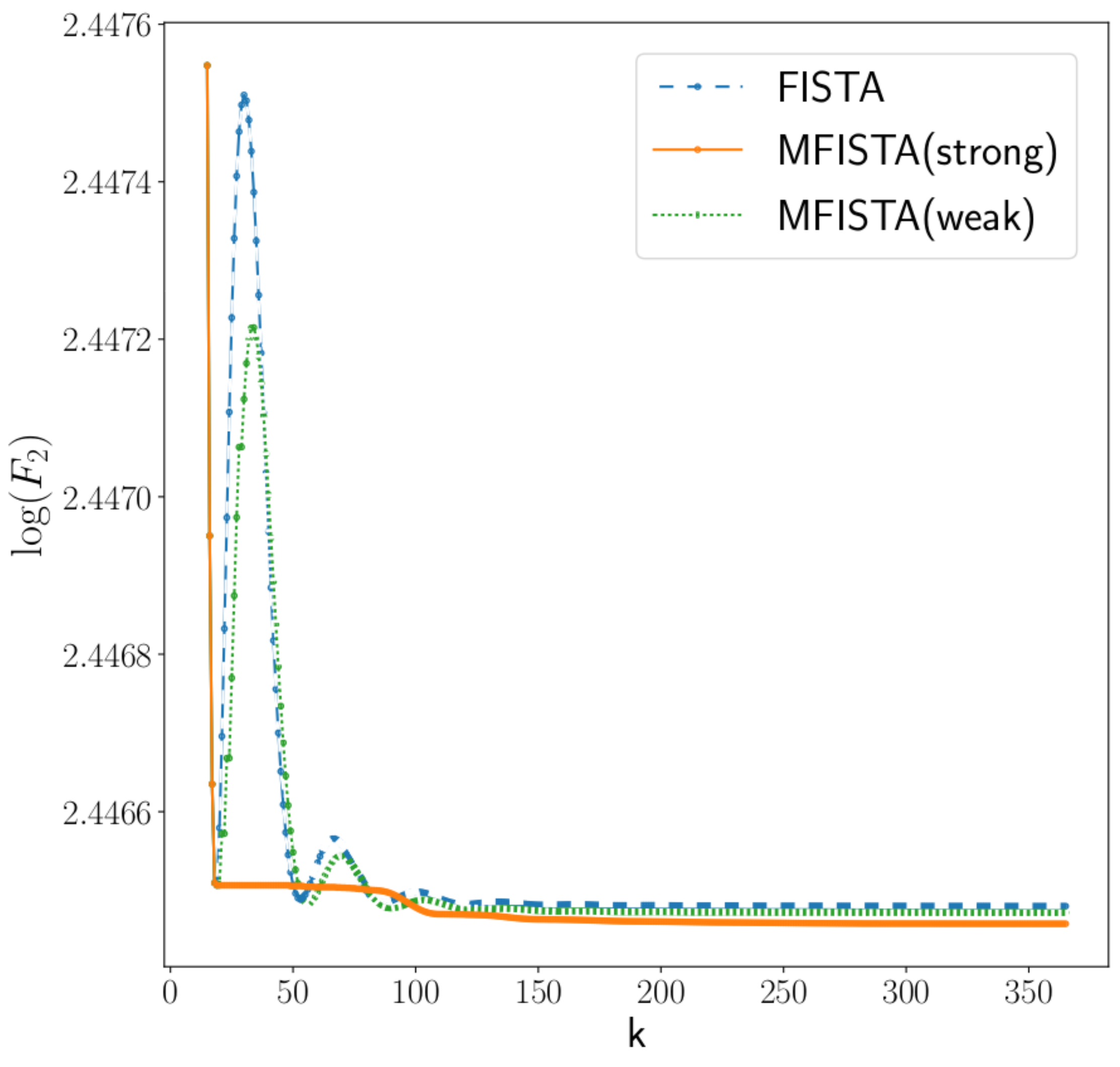}
        $F_2$
      \end{minipage}
      \begin{minipage}[b]{0.33\linewidth}
        \centering
        \includegraphics[keepaspectratio, width=50mm]{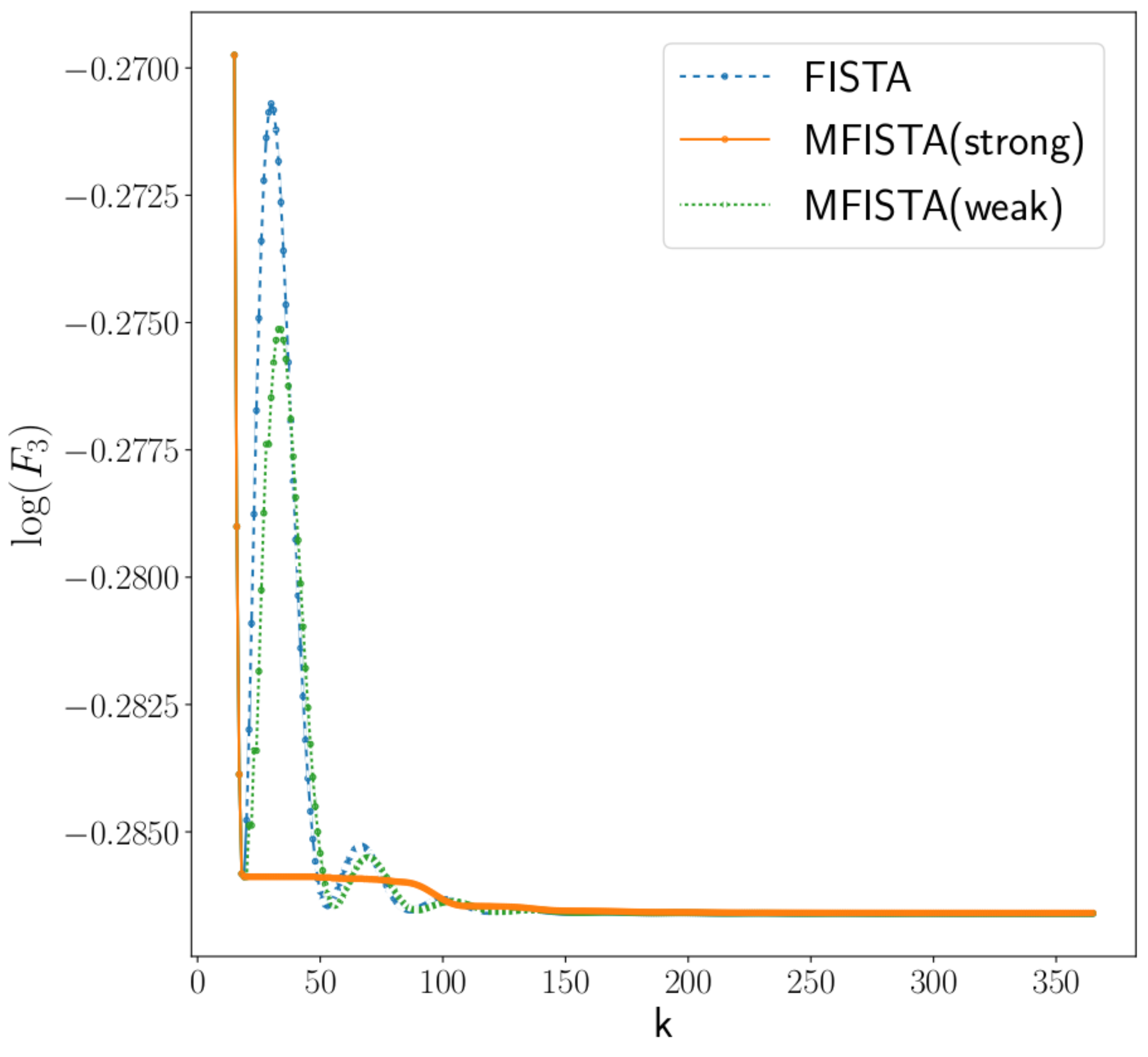}
        $F_3$
      \end{minipage}
    \end{tabular}
    \caption{Functional values for Problem~1}
    \label{fig: FGDS09}
  \end{center}
\end{figure}

\begin{figure}[htbp]
  \begin{center}
    \begin{tabular}{@{\hspace{-3pt}}c} 
      \begin{minipage}[b]{0.33\linewidth}
        \centering
        \includegraphics[keepaspectratio, width=50mm]{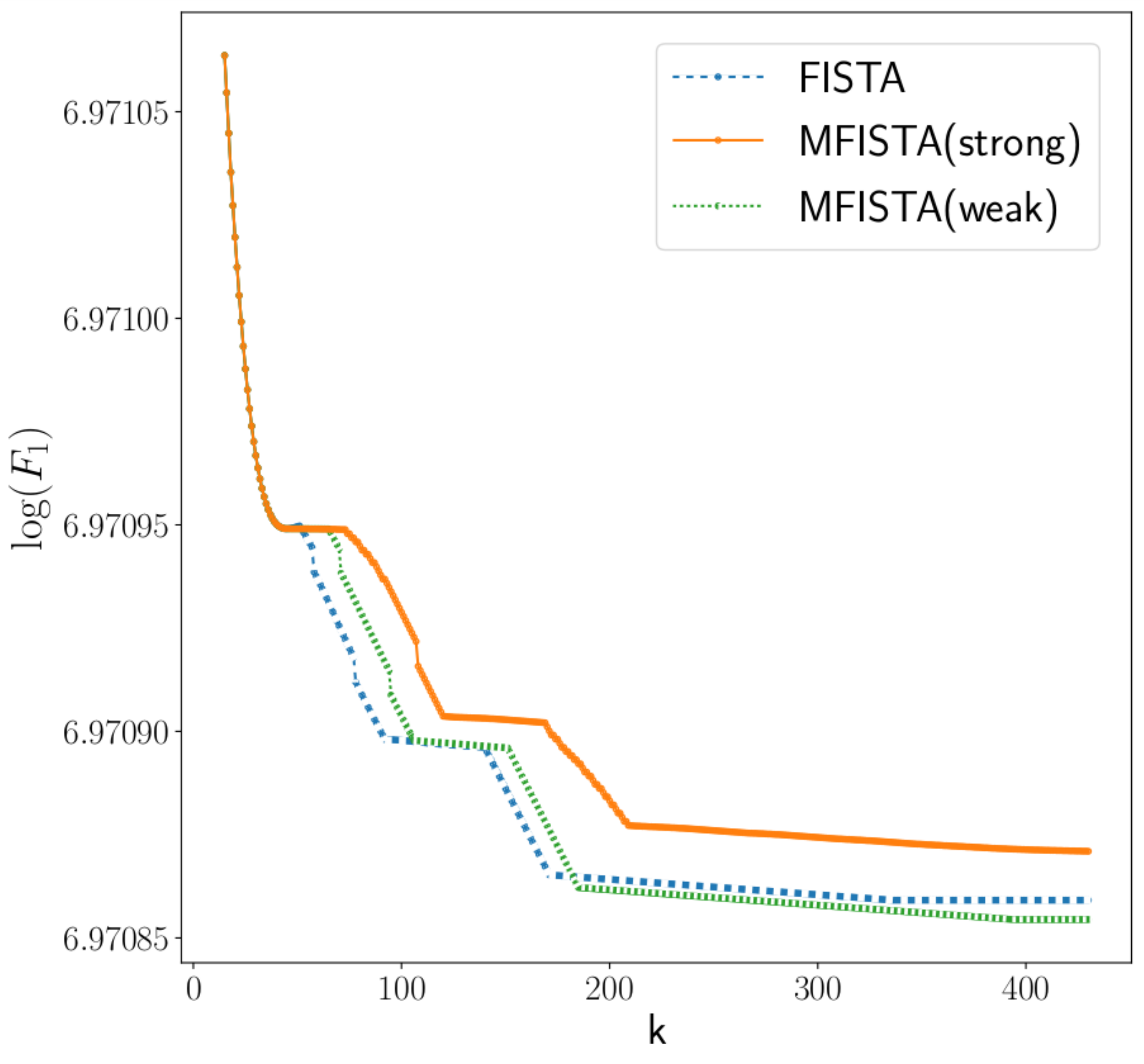}
        $F_1$
      \end{minipage}
      \begin{minipage}[b]{0.33\linewidth}
        \centering
        \includegraphics[keepaspectratio, width=50mm]{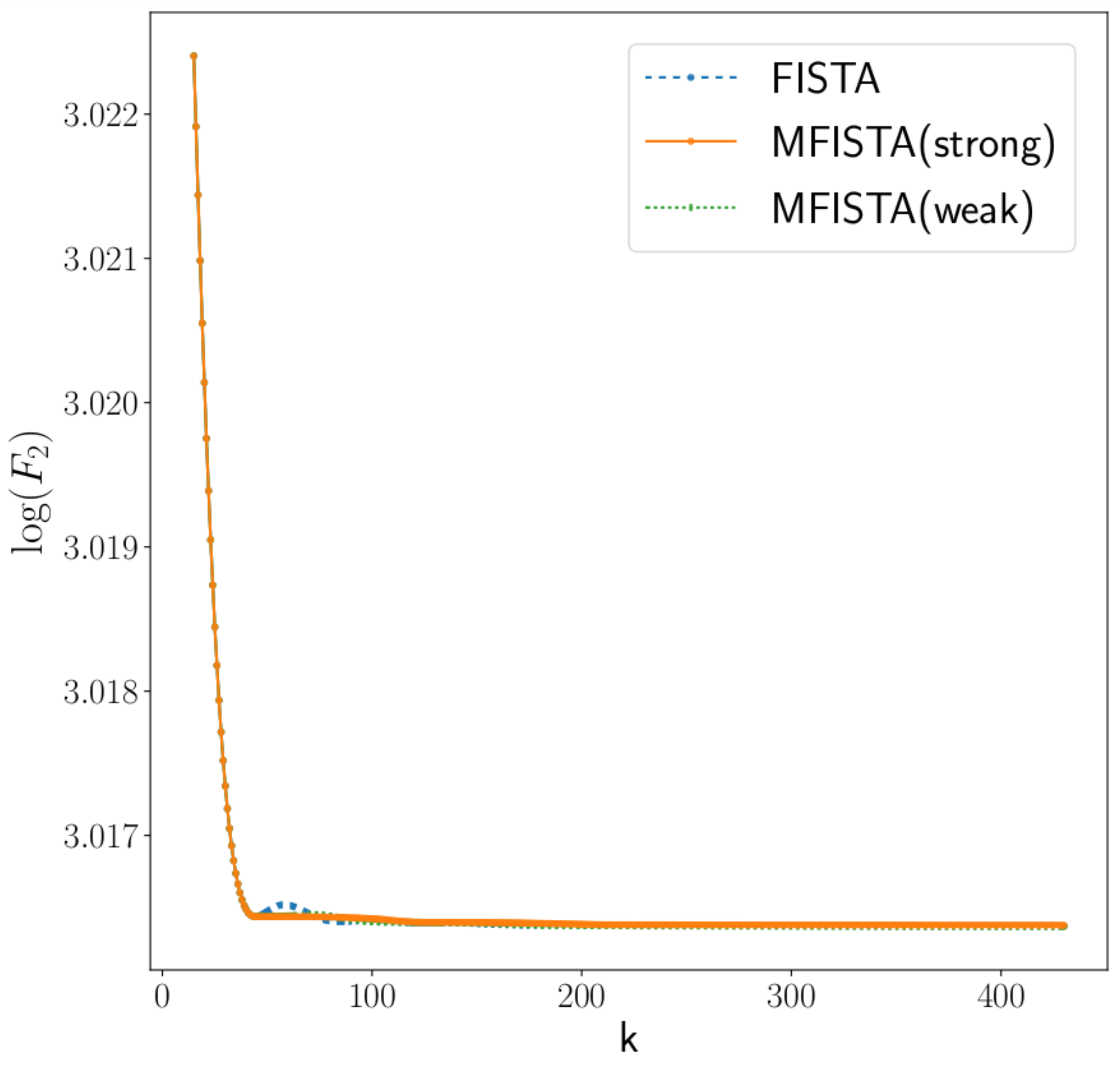}
        $F_2$
      \end{minipage}
      \begin{minipage}[b]{0.33\linewidth}
        \centering
        \includegraphics[keepaspectratio, width=50mm]{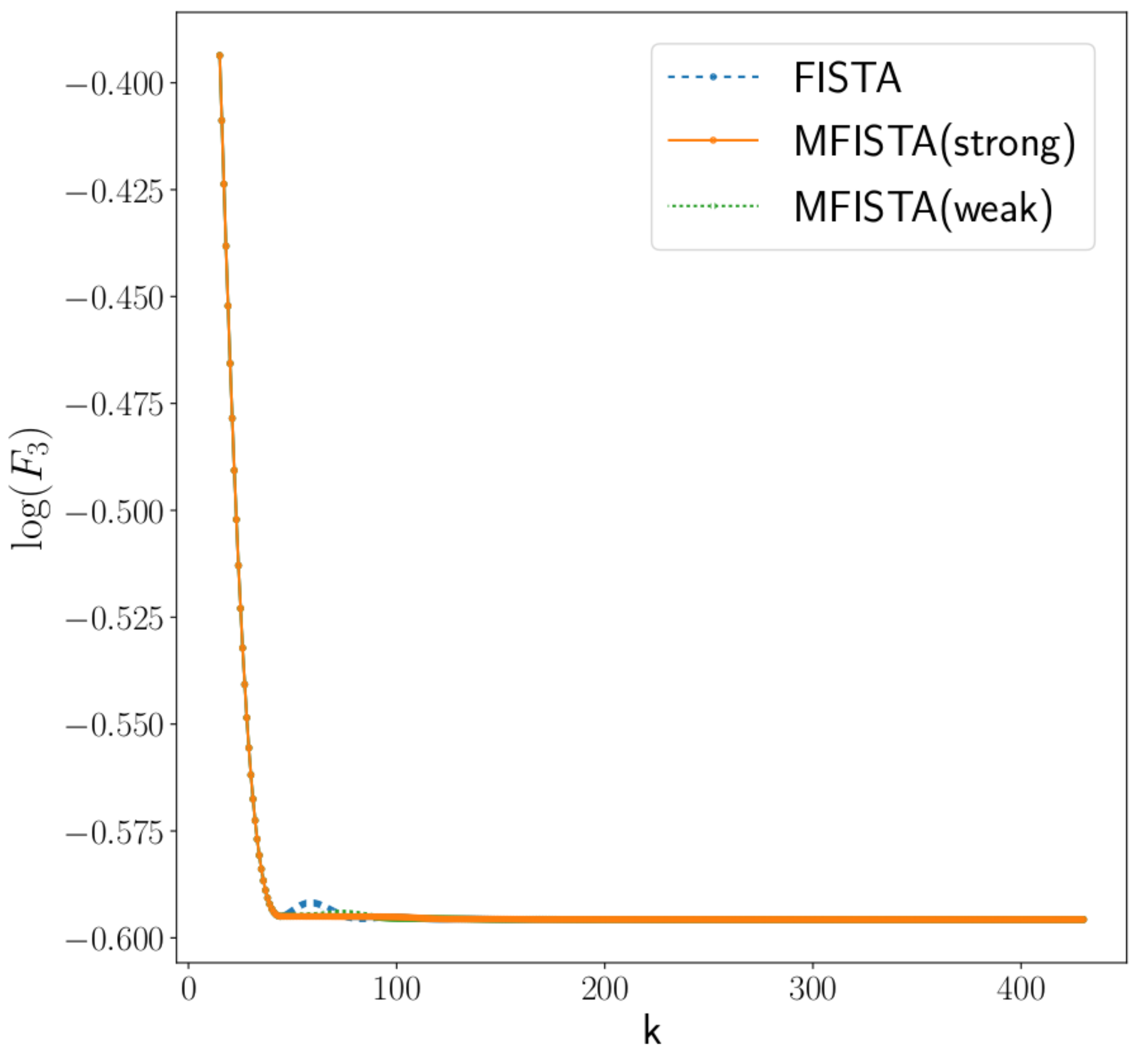}
        $F_3$
      \end{minipage}
    \end{tabular}
    \caption{Functional values for Problem~2}
    \label{fig: FGDS09_cons}
  \end{center}
\end{figure}

\begin{table}
  \caption{Average iterations and time for Problem~1}
  \label{table: iteration}
  \centering
  \begin{tabular}{ccccc}
  \hline
  & \red{PGM} &  FISTA & Weak-MFISTA & Strong-MFISTA \\
  \hline
  Iterations & 606.24 & 206.42 & 203.88 & 202.37 \\
  Time (s)   & 150.85 & 50.41 & 49.78 & 49.52 \\
  \hline
  \end{tabular}
\end{table}

\begin{table}
  \caption{Average iterations and time for Problem~2}
  \label{table: iteration_cons}
  \centering
  \begin{tabular}{ccccc}
  \hline
  & \red{PGM} &  FISTA & Weak-MFISTA & Strong-MFISTA \\
  \hline
  Iterations & 981.31 & 276.91 & 277.42 & 303.46 \\
  Time (s)   & 450.40 & 131.39 & 131.99 & 144.09 \\
  \hline
  \end{tabular}
\end{table}

We also check the average number of iterations and time taken for each
algorithm. As it can be seem in Tables~\ref{table: iteration}
and~\ref{table: iteration_cons}, all accelerated methods are better
than the proximal gradient method. \red{However, by checking also the
  standard deviations, we conclude that FISTA, Weak-MFISTA and
  Strong-MFISTA are similar in terms of iterations and time.} This
result is interesting, since in the single-objective case, we usually
expect MFISTA to spend more time than FISTA, because of the extra
computations.

\red{Now, to observe the difference among the accelerated methods, we 
consider the following problem, which is the image deblurring
problem~\cite{BT09a}, together with an arbitrary objective
function that was only added to make the problem multiobjective.} \\

\noindent \red{\textbf{Problem 3.} $m=2$, and
\begin{align*}
  & f_1(x) = \| BWx - b \|^2, \quad f_2(x) = 0 \\
  & g_1(x) = \lambda \| x\|_1, \quad g_2(x) = \lambda \| x - 1 \|_1
\end{align*}
where $\|\cdot\|_1$ denotes the $\ell_1$-norm, $B$ is the matrix
representing the blur operator, $W$ is the inverse of the Haar wavelet
transform, $b$ is the vectorized observed image, and $\lambda$ is the
regularization parameter.} \\

\red{For Problem~3, we consider the $256 \times 256$ cameraman test image
(see~\cite[Section~5]{BT09a}) that goes through a $9 \times 9$
Gaussian blur with standard deviation~$4$, followed by an additional
white Gaussian noise with zero-mean and standard deviation
$10^{-3}$. The observed image's wavelet transform is used as the
initial point, and we set up $\lambda = 2 \times
10^{-5}$. Figure~\ref{fig: cameraman} shows the objective functions
values along the iterations. For the meaningless $F_2$, the objective
function increases with a small order in both FISTA and Weak-MFISTA.
For the image deblurring related function $F_1$, FISTA diverges
considerably during the process, Strong-MFISTA fails to converge, and
Weak-MFISTA converges in few iterations.}

\begin{figure}[htbp]
  \begin{center}
    \begin{tabular}{@{\hspace{-3pt}}c} 
      \begin{minipage}[b]{0.33\linewidth}
        \centering
        \includegraphics[keepaspectratio, width=50mm]{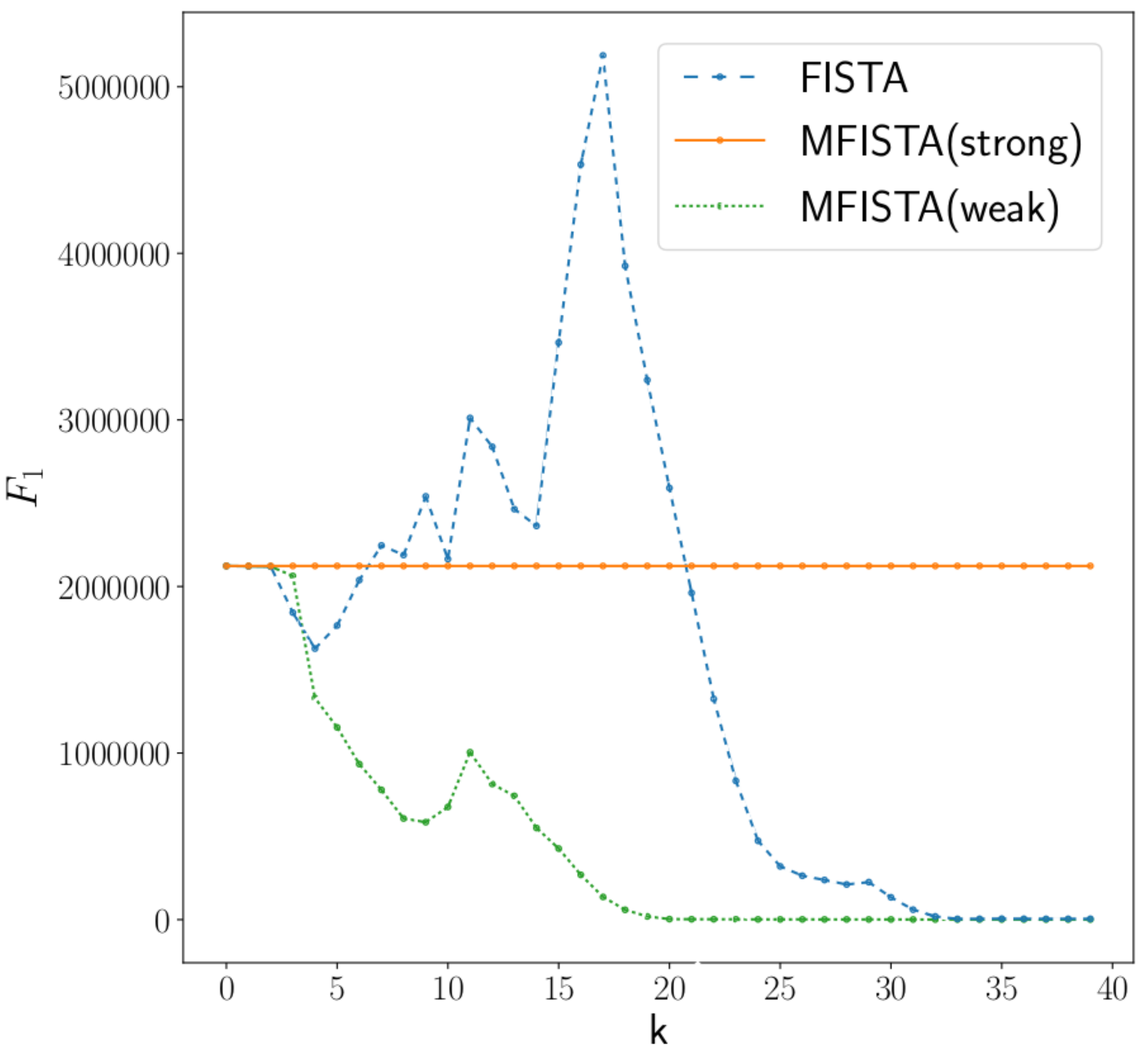}
        $F_1$
      \end{minipage}
      \begin{minipage}[b]{0.33\linewidth}
        \centering
        \includegraphics[keepaspectratio, width=50mm]{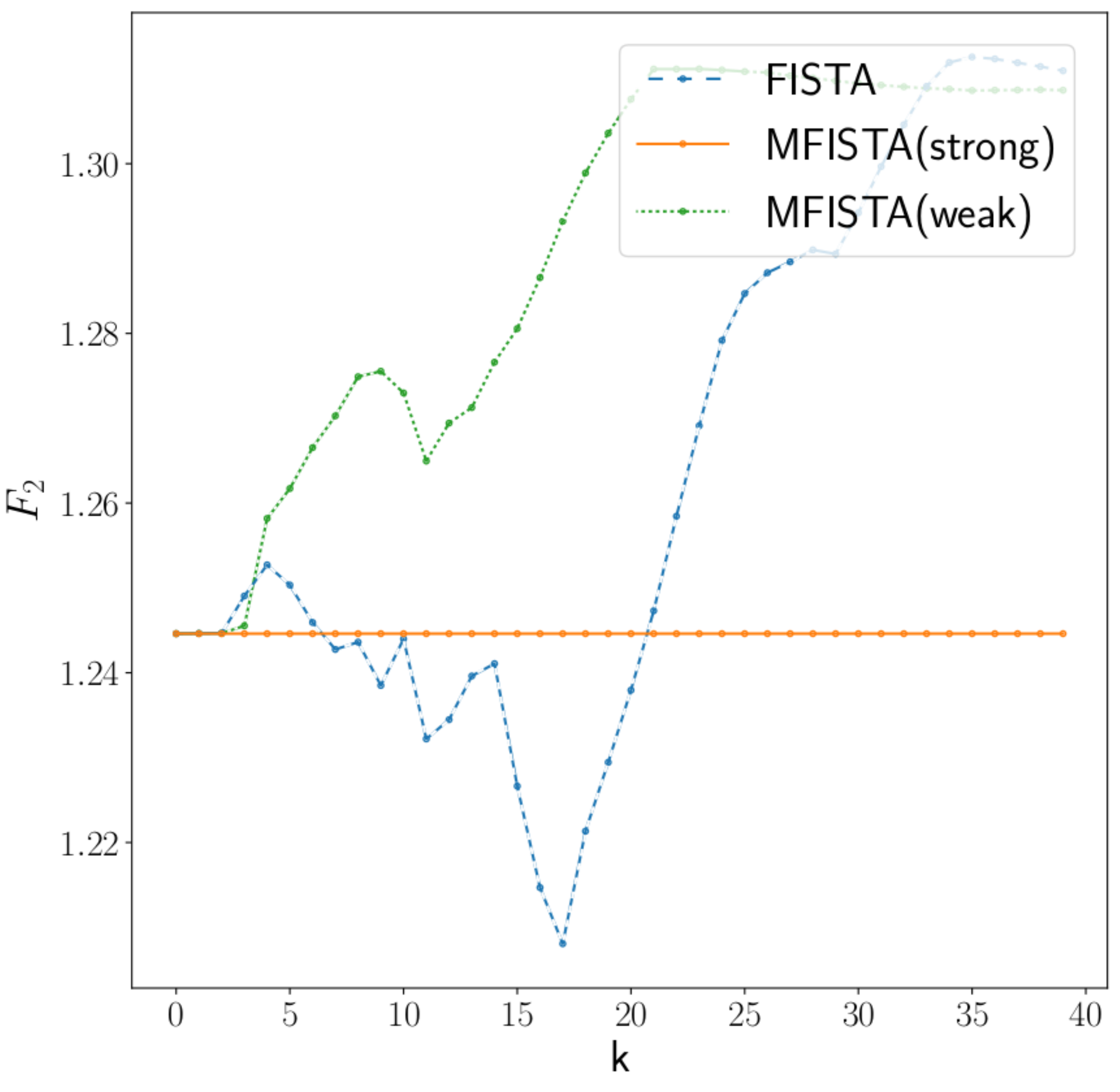}
        $F_2$
      \end{minipage}
    \end{tabular}
    \caption{\red{Functional values for Problem~3}}
    \label{fig: cameraman}
  \end{center}
\end{figure}


\section{Final remarks}
\label{sec:conclusions}

We have proposed an alternative version of the multiobjective
accelerated proximal gradient method, by considering some type of
monotonicity of the objective functions. In particular, imposing a
nonincrease for at least one objective function in each iteration, we
obtain a method that converges globally with rate
$O(1/k^2)$. Moreover, the numerical experiments suggest that the
monotonicity does not interfere in the obtained \red{set of weakly
  Pareto optimal points}, but can be more effective, depending on the
problem. A future work will be to find interesting application
problems, where the method without any monotonicity requirements can
fail. Other restarting techniques should be also studied in the
multiobjective case.\\


\noindent {\bf Acknowledgements.} 
This work was supported by the Grant-in-Aid for
Scientific Research (C) (19K11840 and 21K11769) from Japan Society for
the Promotion of Science.





\printbibliography

\end{document}